\documentclass[10pt,reqno]{amsart}
\usepackage{color}
\usepackage{enumerate}
\usepackage{amsmath,amssymb}

\usepackage{hyperref}

\hypersetup{
    colorlinks=true,
    linkcolor=blue,
    filecolor=magenta,      
    urlcolor=cyan,
}

\newtheorem{theorem}{Theorem}[section]

\newtheorem{corollary}[theorem]{Corollary}

\theoremstyle{remark}
\newtheorem{remark}[theorem]{Remark}

\theoremstyle{definition}
\newtheorem{assumption}[theorem]{Assumption}

\newtheorem{definition}[theorem]{Definition}

\newcommand\cbrk{\text{$]$\kern-.15em$]$}}
\newcommand\opar{\text{\,\raise.2ex\hbox{${\scriptstyle|}$}\kern-.34em$($}}
\newcommand\cpar{\text{$)$\kern-.34em\raise.2ex\hbox{${\scriptstyle |}$}}\,}

\def\qed{{\hfill $\Box$ \bigskip}}

\def\XXint#1#2#3{{\setbox0=\hbox{$#1{#2#3}{\int}$}
\vcenter{\hbox{$#2#3$}}\kern-.5\wd0}}

\newcommand\bZ{\mathbb{Z}}

\newcommand\bE{\mathbb{E}}

\newcommand\bR{\mathbb{R}}

\newcommand\bN{\mathbb{N}}
\newcommand\bP{\mathbb{P}}

\newcommand\fR{\mathbf{R}}

\newcommand\cF{\mathcal{F}}

\newcommand\cS{\mathcal{S}}
\newcommand\cM{\mathcal{M}}

\newcommand\cO{\mathcal{O}}

\newcommand\aint{-\hspace{-0.38cm}\int}

\newcommand{\mysection}[1]{\section{#1}
\setcounter{equation}{0}}

\begin{document}

\title[Degenerate PDEs with lower order terms]
{A weighted $L_q(L_p)$-theory for fully degenerate second-order evolution equations with  unbounded time-measurable coefficients}

\author{Ildoo Kim}
\address{Department of mathematics, Korea university, 1 anam-dong
sungbuk-gu, Seoul, south Korea 136-701}
\email{waldoo@korea.ac.kr}
\thanks{I. Kim has been supported by the National Research Foundation of Korea(NRF) grant funded by the Korea government(MSIT) (No.NRF-2020R1A2C1A01003959)}



\begin{abstract}
We study the   fully degenerate second-order evolution  equation 
\begin{align}
						\label{abs eqn}
u_t=a^{ij}(t)u_{x^ix^j} +b^i(t) u_{x^i} + c(t)u+f, \quad t>0, x\in \bR^d  
\end{align}
given with the zero initial data. Here $a^{ij}(t)$, $b^i(t)$, $c(t)$ are merely locally integrable functions, and $(a^{ij}(t))_{d \times d}$ is a nonnegative symmetric matrix with the smallest eigenvalue $\delta(t)\geq 0$. 
We show that there is a positive constant $N$ such that
\begin{align}
							\notag
& \int_0^{T} \left(\int_{\bR^d} \left(|u|+|u_{xx} |\right)^{p}  dx \right)^{q/p} 
e^{-q\int_0^t c(s)ds} w(\alpha(t))   \delta(t) dt \\
						\label{abs weight}
&\leq N \int_0^{T} \left(\int_{\bR^d} \left|f\left(t,x\right)\right|^{p}  dx \right)^{q/p} e^{-q\int_0^t c(s)ds} w(\alpha(t)) (\delta(t))^{1-q}  dt,
\end{align}
where $p,q \in (1,\infty)$, $\alpha(t)=\int_0^t \delta(s)ds$, and  $w$ is a Muckenhoupt's weight.
\end{abstract}

\maketitle

\mysection{introduction}
Needless to say, the second-order partial differential equations equations with degenerate or unbounded coefficients have been extensively studied for a long time.
To the best of our knowledge, the starting point of this study was 
Keldysh, Fichera, and Oleĭnik's work (see e.g. \cite{keldysh1951some,fichera1963unified,oleinik1965smoothness,oleinik1966alcuni,oleinik2012second}). 
Moreover, it is very popular to study a (maximal regularity) $L_p$-theory and its generalization to $L_q(L_p)$-theory in harmonic analysis, Fourier analysis, and partial differential equations after Calder\'on and Zygmund's work. 
For the historical works and backgrounds of $L_p$-theories and their generalizations, we refer some outstanding books \cite{Krylov2008,ladyvzenskaja1988linear,Stein1993,grafakos2014classical,grafakos2014modern,hytonen2016analysis,hytonen2018analysis}.
These days, there are tons of papers handling degenerate and unbounded coefficients in various prospectives.
Among recent works with various prospectives, we only refer the author to \cite{kim2007sobolev,fornaro2012degenerate,du2013wm,mamedov2013first,fornaro2015second,gerencser2015solvability,leahy2015degenerate,gerencser2016stochastic,gadjiev2017priori,kim2017heat,li2017weighted,pruess2017second,cao2018weighted,monticelli2019poincare,amann2020linear,dong2020parabolic,gadjiev2020solvability,wu2021lp,dong2021parabolic,dong2021regularity,fornaro2022multi,kozhanovinverse,schochet2022sobolev,zulfaliyeva2022smoothness}. These results handle equations having degenerate or unbounded coefficients in Sobolev spaces.

%
%Among these studies, we want to mention some recent results to treat solvability of solutions to second-order (stochastic) partial differential equations in Sobolev spaces, e.g. 

With the degeneracy in the equation, it is hard to expect to obtain full regularity estimates of solutions unless  there are weights involved  in estimates. For instance, by taking the leading coefficients $a^{ij}(t)=0$ for all $i,j,t$, we see that it is not possible to obtain the unweighted 
maximal $L_p$-regularity 
\begin{align}
							\label{cal-zyg est}
\int_0^T \int_{\bR^d} |u_{xx}(t,x)|^p dt dx \leq N \int_0^T \int_{\bR^d} |f(t,x)|^p dt dx.
\end{align}
%for all $f \in L_p( (0,T) \times \bR^d)$ and the corresponding solution $u$ to \eqref{abs eqn}.
%Nonetheless, as shown in \eqref{abs weight}, we can still characterize a regularity relation between the free term and the solution with the help of appropriate weights. 
Hence,  weights have been commonly used to  controls the degeneracy or unboundeness (singularity) of the coefficients. 
However, most results in the literature focus on degeneracy or singularity near the boundary of a domain.
If we consider the whole space, it is naturally not expected that there is a regularity gain of a solution in general due to the extreme case such as
$u_t=f$, which could be understood as one of equation \eqref{abs eqn} with coefficients $a^{ij}(t)=0$ for all $t$.
Hence when it comes to the solvability of second-order equations with degeneracy in the whole space, people used to only prove the existence and uniqueness of a weak solution without considering regularity gain from the equations.

Nonetheless, there is a way to express  an $L_p$-norm of second derivatives of a solution $u$ with a weight which could be singular even in the whole space. 
For instance, assume that the degeneracy happens on a  time interval $(a,b)$, then $\delta(t)=0$ for all $t \in (a,b)$.
Then we cannot expect the smoothing gain from the diffusion equations and  the Sobolev second derivatives $u_{xx}$ fails to exist.
However, since there is the weight $\delta(t)$  in the first line of \eqref{abs weight}, the inequality is still true if we understand the second line of \eqref{abs weight} as an improper integral. To the best of our knowledge, this type estimate is firstly introduced by the author and collaborator in \cite{kim2018second,kim2019sharp}. In this paper, we add Muckenhoupt's weights in estimates and extend $L_p$-estimates to $L_q(L_p)$-estimates with lower-order terms.

It is well-known that probabilistic methods are very powerfully working for leading coefficients which are unbounded and have degeneracy (cf. \cite{krylov1995introduction,cerrai2001}).
We remark that probabilistic tools play very important roles to obtain our results.
Especially, to obtain \eqref{abs weight}, it requires to understand the relation among the constant $N$, the degeneracy, and the unboundedness of coefficients $a^{ij}(t)$. 
Maximal $L_p$-regularity estimates such as \eqref{cal-zyg est} originally came from $L_p$-boundedness of singular integral operators.
However, the exact relation among parameters related to coefficients is hard to obtain from singular integral theories since all parameters are combined in a complicated way to control  singularities of operators. 
We found that this relation could be more clear by applying probabilistic representations of solutions (see Theorem \ref{enhance est thm}).

We believe that our result could initiate various interesting weighted estimates for degenerate second-order equations with space dependent coefficients 
or domain problems.

This paper is organized as follows. In Section \ref{section main}, we introduce  our main results.
A probabilistic solution representation and its application to estimate a solution $u$ with general weights are given in Section \ref{section representation}
Weighted estimates for non-degenerate equations  are shown in Section \ref{section non}.
Finally, the proof of the main theorem is specified in Section \ref{pf main thm}.

We finish the introduction with  notation used in the article. 
\begin{itemize}

\item We use Einstein's summation convention throughout this paper. 

\item $\bN$ and $\bZ$ denote the natural number system and the integer number system, respectively.
As usual $\fR^{d}$
stands for the Euclidean space of points 
$$
x=
\begin{pmatrix}
x^1\\
x^2\\
\vdots \\
x^d
\end{pmatrix}.
$$
Frequently, the coordinates of the vector $x$ is denoted in a row form, i.e. $x=(x^1,\ldots,x^d)$. 
We use the notation $(a^{ij})_{d \times d}$ to denote the $d$ by $d$ matrix whose entry in $i$-th row and $j$-th column is $a^{ij}$.
 For $i=1,...,d$, multi-indices $\alpha=(\alpha_{1},...,\alpha_{d})$,
$\alpha_{i}\in\{0,1,2,...\}$, and functions $u(x)$ we set
$$
u_{x^{i}}=\frac{\partial u}{\partial x^{i}}=D_{i}u,\quad
D^{\alpha}u=D_{1}^{\alpha_{1}}\cdot...\cdot D^{\alpha_{d}}_{d}u.
$$
%We also use the notation $D^m$ for a partial derivative of order $m$ with respect to $x$.  
%Sometimes the notation $D^m_x$ is used to denote the variable instead of $D^m$. 
%

\item $C^\infty(\bR^d)$ denotes the space of infinitely differentiable functions on $\bR^d$. 
$\cS(\bR^d)$ is the Schwartz space consisting of infinitely differentiable and rapidly decreasing functions on $\bR^d$.
By $C_c^\infty(\bR^d)$,  we denote the subspace of $C^\infty(\bR^d)$ with the compact support.

\item  For $n\in \bN$ and $\cO\subset \bR^d$ and a normed space $F$, 
by $C(\cO;F)$,  we denote the space of all $F$-valued continuous functions $u$ on $\cO$ having $|u|_{C}:=\sup_{x\in O}|u(x)|_F<\infty$.
%If $F=\fR$, then we simply put $C^n(\cO)=C^n(\cO;\fR)$.

\item For $p \in [1,\infty)$, a normed space $F$,
%by $L_{p}(\mathcal{O};F)$ we denote the set
%of $F$-valued Lebesgue measurable function $u$ on $\mathcal{O}$ satisfying
%\[
%\left\Vert u\right\Vert _{L_{p}(\mathcal{O};F)}:=\left(\int_{\mathcal{O}}\|u(x)\|_{F}^{p}dx\right)^{1/p}<\infty.
%\]
%We write $L_{p}(\mathcal{O})=L_{p}(\mathcal{O};\fR)$ and $L_{p}=L_{p}(\fR^{d})$.
and a  measure space $(X,\mathcal{M},\mu)$, 
by $L_{p}(X,\cM,\mu;F)$,
we denote the space of all $F$-valued $\mathcal{M}^{\mu}$-measurable functions
$u$ so that
\[
\left\Vert u\right\Vert _{L_{p}(X,\cM,\mu;F)}:=\left(\int_{X}\left\Vert u(x)\right\Vert _{F}^{p}\mu(dx)\right)^{1/p}<\infty,
\]
where $\mathcal{M}^{\mu}$ denotes the completion of $\cM$ with respect to the measure $\mu$. 
%For $p=\infty$, we write $u \in L_{\infty}(X,\cM,\mu;F)$ if
%$$
% \|u\|_{L_{\infty}(X,\cM,\mu;F)} 
%:= \inf\left\{ \nu \geq 0 : \mu( \{ x: \|u(x)\|_F > \nu\})=0\right\} <\infty.
%$$
%where $\sup$ denotes the essential supremum with respect to measure $\mu$.
If there is no confusion for the given measure and $\sigma$-algebra, we usually omit them.
%In particular, we set $L_p=L_p(\bR^d)$, which denotes the $L_p$-space on $\bR^d$ with the Lebesuge measurable sets and Lebesgue measure. 
%\item
%For  functions $f(x)$ and $g(x)$, we use the notation $f=O(g)$ as $x \to x_0$ iff
%there exist positive constant $\delta$ and $N$ such that
%$$
%|f(x)| \leq N|g(x)|
%$$
%for all $|x-x_0| < \delta$.

\item For  measurable set  $\cO \subset \bR^d$, $|\cO|$ denotes the Lebesgue measure of $\cO$.

%\item For a $d \times d$ matrix $A$, we use the notation $A=(a^{ij})$ ($i,j=1,\ldots,d)$ to emphasize the components of the matrix $A$. Here $a^{ij}$ denotes the value at $i$-th row and $j$-th column. 

\item By $\cF$ and $\cF^{-1}$ we denote the d-dimensional Fourier transform and the inverse Fourier transform, respectively. That is,
$\cF[f](\xi) := \int_{\fR^{d}} e^{-i x \cdot \xi} f(x) dx$ and $\cF^{-1}[f](x) := \frac{1}{(2\pi)^d}\int_{\fR^{d}} e^{ i\xi \cdot x} f(\xi) d\xi$.

\item 
We write $a \lesssim b$ if there is a positive constant $N$ such that $ a\leq N b$.
The constant $N$ may change from a location to a location, even within a line. 
If we write $N=N(a,b,\cdots)$, this means that the
constant $N$ depends only on $a,b,\cdots$. 
The dependence of the constant $N$ is usually specified in the statements of theorems, lemmas, and corollaries.

\end{itemize}

%\mysection{main result}

%\vspace{5mm}

%We use the Banach spaces
%introduced in \cite{KK2} and \cite{Lo2}.

\mysection{Setting and main result}
									\label{section main}
	
Throughout the paper, we fix $d \in \bN$ to denote the dimension of the space variable and all functions are real-valued if there is no special comment. 
We study the following degenerate second-order evolution equation
\begin{align}
				\notag
&u_t(t,x)=a^{ij}(t)u_{x^ix^j}(t,x) +b^i(t) u_{x^i}(t,x) + c(t)u(t,x)+f(t,x),   \\
&u(0,x)=0, 
 \qquad \qquad \qquad \qquad \qquad \qquad \qquad \qquad (t,x) \in (0,T) \times \bR^d.
				\label{main eqn}
\end{align}
We  emphasize that our coefficients $a^{ij}(t)$, $b^i(t)$, and $c(t)$ do not satisfy any regularity conditions.
More importantly, our coefficients $a^{ij}(t)$, $b^i(t)$, and $c(t)$ can be unbounded and degenerate.
Here are more concrete conditions on the coefficients $a^{ij}(t)$, $b^i(t)$, and $c(t)$.

\begin{assumption}
						\label{main as}
\begin{enumerate}[(i)]
\item Assume that  there exists a measurable mapping $\delta(t)$ from $(0,\infty)$ to $[0,\infty)$ such that
\begin{align*}
a^{ij}(t) \xi^i \xi^j \geq \delta(t) |\xi|^2 \quad \forall t \in [0,\infty)~ \text{and}~ \xi \in \bR^d. 
\end{align*}
\item Assume that the coefficients $a^{ij}(t)$, $b^i(t)$, and $c(t)$ are locally integrable, i.e. 
\begin{align}
							\label{coe local inte}
\int_0^T \left(|a^{ij}(t)|+|b^i(t)|+|c(t)|\right) dt < \infty \qquad \forall T \in (0,\infty)~\text{and}~ \forall i,j.
\end{align}
\end{enumerate}
\end{assumption}

For $ T \in (0,\infty)$ and a measurable function $u$ on $(0,T) \times \bR^d $, we say that $u$ is locally integrable if
\begin{align*}
\int_0^t \int_{|x| <c} |u(t,x)| dx dt < \infty \quad \forall t \in (0,T)~\text{and}~ \forall c>0.
\end{align*}

\begin{definition}[Solution]
							\label{def sol}
Let $T \in (0,\infty)$ and $f$ be a locally integrable function on $(0,T) \times \bR^d$. 
We say that a locally integrable function $u$ is a solution to \eqref{main eqn} if for any $\varphi \in C_c^\infty(\bR^d)$,
\begin{align}
						\notag
(u(t,\cdot), \varphi) 
&= \int_0^t \left(u(s,\cdot), a^{ij}(s)\varphi_{x^ix^j}+b^i(s)\varphi_{x^i} +c(s) \varphi \right) ds \\
						\label{weak formulation}
&\quad +\int_0^t \left(f(s,\cdot), \varphi \right) ds
\quad \forall t \in (0,T),
\end{align}
where $(u(t,\cdot), \varphi)$ denotes the $L_2(\bR^d)$-inner product, i.e.
$$
(u(t,\cdot), \varphi):= \int_{\bR^d} u(t,x)\varphi(x)dx.
$$
\end{definition}
\begin{remark}
Due to the definition of a solution, it is obvious that
$$
a^{ij}(t) u_{x^ix^j} = \frac{a^{ij}(t) + a^{ji}(t)}{2} u_{x^ix^j}.
$$
Thus without loss of generality, we may assume that our coefficient matrix $(a^{ij}(t))_{d \times d}$ is nonnegative symmetric for all $t$.
Additionally, $\delta(t)$ in Assumption \ref{main as}(i) can be chosen by the smallest eigenvalue of $(a^{ij}(t))_{d \times d}$.
\end{remark}

We recall the definition of Muckenhoupt's weights.
\begin{definition}[Muckenhoupt's weight]
							\label{def weight}
For $q\in(1,\infty)$, let $A_q(\bR)$ be the class of all nonnegative and locally integrable functions $w$ on $\bR$ satisfying
$$
[w]_{A_q(\bR)}:=\sup_{ -\infty<a<b <\infty}\left(\aint_{(a,b)}w(t)dt\right)\left(\aint_{(a,b)}w(t)^{-1/(q-1)}dt\right)^{q-1}<\infty,
$$
where 
$$
\aint_{(a,b)}w(t)dt = \frac{\int_a^b w(t)dt}{b-a}.
$$
\end{definition}

Finally, we introduce our main result.
\begin{theorem}
						\label{main thm}
Let $T \in (0,\infty)$, $p,q \in (1,\infty)$, and  $w \in A_q (\bR)$.
Suppose that Assumption \ref{main as} holds. 
Then for any  locally integrable function $f$ on $(0,T) \times \bR^d$, there is a unique solution $u$ to equation \eqref{main eqn} such that
\begin{align}
						\notag
&  \sup_{t \in [0,T]} \left[ \left(\int_{\bR^d} \left| u\left(t,x\right)(t,x)\right|^{p}  dx \right)^{q/p} e^{-q\int_0^t c(s)ds} \right] \\
						\label{main est -1}
&\leq \left[\int_0^{\alpha(T)} w(t)^{-\frac{1}{q-1}} dt\right]^{q-1} \int_0^{T} \left(\int_{\bR^d} \left|f\left(t,x\right)\right|^{p}  dx \right)^{q/p}e^{-q\int_0^t c(s)ds}  w(\alpha(t)) |\delta(t)|^{1-q}  dt,
\end{align}

\begin{align}
						\notag
& \int_0^{T} \left(\int_{\bR^d} \left| u\left(t,x\right)(t,x)\right|^{p}  dx \right)^{q/p} 
e^{-q\int_0^t c(s)ds} w(\alpha(t))   \delta(t) dt \\
						\label{main est 0}
&\leq  [w]_{A_q(\bR)} [\alpha(T)]^q  \int_0^{T} \left(\int_{\bR^d} \left|f\left(t,x\right)\right|^{p}  dx \right)^{q/p} e^{-q\int_0^t c(s)ds} w(\alpha(t)) |\delta(t)|^{1-q}  dt,
\end{align}
and
\begin{align}
						\notag
& \int_0^{T} \left(\int_{\bR^d} \left| u_{xx}\left(t,x\right)(t,x)\right|^{p}   dx \right)^{q/p} 
e^{-q\int_0^t c(s)ds} w(\alpha(t))   \delta(t) dt \\
						\label{main est}
&\leq N \int_0^{T} \left(\int_{\bR^d} \left|f\left(t,x\right)\right|^{p}  dx \right)^{q/p}e^{-q\int_0^t c(s)ds} w(\alpha(t)) |\delta(t)|^{1-q}  dt, 
\end{align}
where $\alpha(t) = \int_0^t \delta(s)ds$ and  $N$ is a positive constant depending only on $d$, $p$, $q$, and $[w]_{A_q(\bR)}$.
In particular, for any  $-1 < \beta < q-1$,
\begin{align}
						\notag
&  \sup_{t \in [0,T]} \left[ \left(\int_{\bR^d} \left| u\left(t,x\right)(t,x)\right|^{p}  dx \right)^{q/p} e^{-q\int_0^t c(s)ds} \right] \\
						\notag
&\leq \left[ \frac{q-1}{q-1-\beta}\right]^{q-1} \left[\int_0^T \delta(t)dt\right]^{q-1-\beta}  \\
						\label{main particular est -1}
&\quad \times \int_0^{T}\left(\int_{\bR^d} \left|f\left(t,x\right)\right|^{p}  dx \right)^{q/p} e^{-q\int_0^t c(s)ds} \left| \int_0^t \delta(s)ds \right|^{\beta}  (\delta(t))^{1-q}  dt,
\end{align}
\begin{align}
						\notag
& \int_0^{T} \left(\int_{\bR^d} \left| u\left(t,x\right)(t,x)\right|^{p}  \right)^{q/p} 
e^{-q\int_0^t c(s)ds} \left| \int_0^t \delta(s)ds \right|^{\beta}   \delta(t) dt \\
							\notag
&\leq   [|t|^\beta]_{A_p(\bR)}\left[ \int_0^T \delta(t)dt \right]^q \\
						\label{main particular est 0}
&\quad \times \int_0^{T} \left(\int_{\bR^d} \left|f\left(t,x\right)\right|^{p}  dx \right)^{q/p} e^{-q\int_0^t c(s)ds} \left| \int_0^t \delta(s)ds \right|^{\beta}  (\delta(t))^{1-q}  dt,
\end{align}
and
\begin{align}
						\notag
& \int_0^{T} \left(\int_{\bR^d} \left| u_{xx}\left(t,x\right)(t,x)\right|^{p}  \right)^{q/p} 
e^{-q\int_0^t c(s)ds} \left| \int_0^t \delta(s)ds \right|^{\beta}   \delta(t) dt \\
						\label{main particular est}
&\leq   N\int_0^{T} \left(\int_{\bR^d} \left|f\left(t,x\right)\right|^{p}  dx \right)^{q/p} e^{-q\int_0^t c(s)ds} \left| \int_0^t \delta(s)ds \right|^{\beta}  (\delta(t))^{1-q}  dt,
\end{align}
where $N$ depends only on $d$, $p$, $q$, and $\beta$.
\end{theorem}
A proof of Theorem \ref{main thm} is given in Section \ref{pf main thm}.
\vspace{2mm}

\begin{remark}
\begin{enumerate}[(i)]
\item For $t \in [0,T]$ so that $\delta(t)=0$, the existence of the Sobolev derivatives $u_{xx}(t,x)$ is not guaranteed by \eqref{main est}.
Moreover, $\delta(t)$ can be zero on a set with a positive Lebesgue measure, which is far from Muckenhoupt's weight.
\item Since $\delta(t)$ can be zero on a set with a positive measure, the integral 
$$
\int_0^{T} \left(\int_{\bR^d} \left|f\left(t,x\right)\right|^{p}  dx \right)^{q/p} e^{-q\int_0^t c(s)ds} w(\alpha(t)) (\delta(t))^{1-q}  dt
$$
is understood in an improper sense, i.e.
\begin{align}
						\notag
&\int_0^{T} \left(\int_{\bR^d} \left|f\left(t,x\right)\right|^{p}  dx \right)^{q/p} e^{-q\int_0^t c(s)ds} w(\alpha(t)) (\delta(t))^{1-q}  dt \\
						\label{improper integ}
&=\lim_{\varepsilon \downarrow 0}
\int_0^{T} \left(\int_{\bR^d} \left|f\left(t,x\right)\right|^{p} dx \right)^{q/p}e^{-q\int_0^t c(s)ds} w(\alpha(t) + \varepsilon t) (\delta(t)+\varepsilon)^{1-q}  dt.
\end{align}
\item If
$$
\int_0^T  \left(\int_{\bR^d} |f(s,x)|^p  dx \right)^{q/p} e^{-q\int_0^t c(s)ds} w(\alpha(s)) |\delta(s)|^{1-q} ds  < \infty,
$$
then the local integrability condition on $f$ is not necessary in Theorem \ref{main thm}.
In other words,  the finiteness condition implies the local integrability of $f$. 
To investigate this fact, let $ t \in (0,T)$ and $c>0$. 
Then for any $\varepsilon \in (0,1)$, applying H\"older's inequality and the change of variable $\alpha(t)+\varepsilon t \to t$, 
we have
\begin{align}
						\notag
&\int_0^t \int_{|x| <c} |f(t,x)| dx ds \\
						\notag
&=\int_0^t\int_{\bR^d} |f(t,x)|   1_{|x|<c} dx  1_{0<s<t}ds \\
						\notag
&\leq  N\int_0^t  \left(\int_{\bR^d} |f(t,x)|^p  dx \right)^{1/p}  1_{0<s<t} ds \\
						\notag
&\leq  N \left[\int_0^t  \left(\int_{\bR^d} |f(s,x)|^p  dx \right)^{q/p} e^{-q\int_0^s c(\rho)d\rho} w(\alpha(s)+\varepsilon s) |\delta(s) + \varepsilon|^{1-q} ds \right]^{1/q} \\
						\notag
&\quad \times \left[\int_0^t e^{ \frac{q}{q-1}\int_0^s c(\rho)d\rho} w^{-\frac{1}{q-1}}(\alpha(s) + \varepsilon s) (\delta(s) +\varepsilon)  ds \right]^{(q-1)/q} \\
						\label{eqn 20220916 10}
&\leq  N \left[\int_0^t  \left(\int_{\bR^d} |f(s,x)|^p  dx \right)^{q/p} w(\alpha(s)+\varepsilon s) |\delta(s) + \varepsilon|^{1-q} ds \right]^{1/q} \\
						\notag
&\quad \times e^{ \frac{q}{q-1}\int_0^t |c(s)|ds} \left[\int_0^{\alpha(t)+ t}  w^{-\frac{1}{q-1}}(s)  ds \right]^{(q-1)/q}.
\end{align}
It is obvious that $e^{ \frac{q}{q-1}\int_0^t |c(s)|ds} < \infty$ since the function $c(t)$ is locally integrable. 
Moreover, since $w \in A_p(\bR)$,  $\int_0^{\alpha(t)+ t}  w^{-\frac{1}{q-1}}(s)  ds$ is finite.
Therefore, taking $\varepsilon \to 0$ in \eqref{eqn 20220916 10}, (formally) we obtain the local integrability of $f$.

\item 
\eqref{main est -1} and \eqref{main est 0} hold even for $p=1$ or $p=\infty$ (see Corollary \ref{cor 20220916 01}).
Moreover, it is easy to check that \eqref{main est -1} is a stronger estimate than \eqref{main est 0} with a help of the definition of Muckenhoupt's weight, i.e. \eqref{main est -1} implies \eqref{main est 0}.
However, \eqref{main est 0} can be slightly improved by using the probabilistic representation of a solution in the sense that
it cannot be obtained from \eqref{main est -1} directly in general. 
Indeed, formally using \eqref{strong u est 2} with
$$
h_1(t)= w(\alpha(t)) \delta(t)
$$
and
$$
h_2(t)= w(\alpha(t)) |\delta(t)|^{1-q},
$$
we have
\begin{align*}
&\int_0^T \|u(t,\cdot)\|^q_{L_p}e^{-q\int_0^t c(s)ds}  w(\alpha(t)) \delta(t) dt  \\
&\leq \int_0^T \Bigg[ w(\alpha(t)) \delta(t)  \left[ \int_0^t |w(\alpha(s)) |\delta(s)|^{1-q}|^{-\frac{1}{q-1}} ds \right]^{q-1} \\
&\qquad \times \int_0^t \|f(s,\cdot)\|^q_{L_p} e^{-q\int_0^t c(s)ds} w(\alpha(s)) |\delta(s)|^{1-q}ds \Bigg] dt.
\end{align*}

\item 
Obviously, we can obtain 
\begin{align*}
& \int_0^{T} \left(\int_{\bR^d} \left| u_{xx}\left(t,x\right)(t,x)\right|^{p}   w_0(x)dx \right)^{q/p} 
e^{-q\int_0^t c(s)ds} w(\alpha(t))   \delta(t) dt \\
&\leq N \int_0^{T} \left(\int_{\bR^d} \left|f\left(t,x\right)\right|^{p} w_0(x) dx \right)^{q/p} e^{-q\int_0^t c(s)ds} w(\alpha(t)) |\delta(t)|^{1-q}  dt
\end{align*}
if $w_0(x)$ is bounded both below and above.
However, if $w(x)$ has a degeneracy or a singularity (unboundedness), then we believe that it is impossible to add $w_0 \in A_p(\bR^d)$ in the estimates.  In other words, generally, it is not expected to find a positive constant $N$ such that
\begin{align}
								\notag
& \int_0^{T} \left(\int_{\bR^d} \left| u_{xx}\left(t,x\right)(t,x)\right|^{p}   w_0(x)dx \right)^{q/p} 
e^{-q\int_0^t c(s)ds} w(\alpha(t))   \delta(t) dt \\
								\label{eqn 20220917 01}
&\leq N \int_0^{T} \left(\int_{\bR^d} \left|f\left(t,x\right)\right|^{p} w_0(x) dx \right)^{q/p} e^{-q\int_0^t c(s)ds} w(\alpha(t)) |\delta(t)|^{1-q}  dt.
\end{align}
To claim it, assume that \eqref{eqn 20220917 01} holds with $b(t)=c(t)=0$ for all $t$. 
Then we get
\begin{align}
								\notag
& \int_0^{T} \left(\int_{\bR^d} \left| u_{xx}\left(t,x\right)(t,x)\right|^{p}   w_0(x)dx \right)^{q/p} 
 w(\alpha(t))   \delta(t) dt \\
								\label{eqn 20220917 02}
&\leq N \int_0^{T} \left(\int_{\bR^d} \left|f\left(t,x\right)\right|^{p} w_0(x) dx \right)^{q/p} w(\alpha(t)) |\delta(t)|^{1-q}  dt.
\end{align}
Then the function $v(t,x) = u\left(t,x + \int_0^t b(s)ds \right)$ becomes a solution to 
\begin{align*}
&v_t(t,x)=a^{ij}(t)v_{x^ix^j}(t,x) +b^i(t) v_{x^i}(t,x)+f\left(t,x+ \int_0^t b(s)ds \right),   \\
&v(0,x)=0, 
 \qquad \qquad \qquad \qquad \qquad \qquad \qquad \qquad (t,x) \in (0,T) \times \bR^d.
\end{align*}
Thus by \eqref{eqn 20220917 02} with $p=q=2$ and $\delta(t)=1$, we obtain
\begin{align*}
& \int_0^{T} \int_{\bR^d} \left| v_{xx}\left(t,x\right)(t,x)\right|^{2}   w_0\left(x+\int_0^t b(s)ds \right) w(t)dx 
  dt \\
&\leq N \int_0^{T} \int_{\bR^d} \left|f\left(t,x\right)\right|^{2} w_0\left(x+\int_0^t b(s)ds \right)w(t) dx     dt.
\end{align*}

Observe that 
$$
(t,x) \mapsto w(t)w_0\left(x+\int_0^t b(s)ds \right)  \notin A_2(\bR^{d+1})
$$
unless $\int_0^t b(s)ds$ is a constant vector uniformly for all $t$ since $w$ has a singularity or a degeneracy in general. 
Therefore we cannot expect \eqref{eqn 20220917 01} if there is a non-trivial coefficient $b(t)$ in the equation.
Moreover, our main tool is the probabilistic solution representation such as \eqref{sol re}.
We use the translation invariant property of $L_p$-norms with this representation in many parts of proofs of the main theorem.
Thus \eqref{eqn 20220917 01} is impossible to obtain by our method even for the case $b(t)=0$ for all $t$ (see Remark \ref{fail method}).

\item All constants in estimates in Theorem \ref{main thm} do not depend on the integrals of coefficients $a^{ij}$, $b^i$, and $c$.
Thus for a fixed time $T \in (0, \infty)$,  the integrability condition on the coefficients \eqref{coe local inte} can be relaxed to
\begin{align*}
\int_0^t \left(|a^{ij}(t)|+|b^i(t)|+|c(t)|\right) dt < \infty \qquad \forall t \in (0,T)~\text{and}~ \forall i,j.
\end{align*}
\end{enumerate}
\end{remark}

\mysection{Probabilistic solution representations}
									\label{section representation}

In this section, we consider equations without lower-order terms first, i.e.
\begin{align}
							\notag
&u_t(t,x) 
=  a^{ij}(t)u_{x^ix^j}(t,x) +f(t, x)  \qquad (t,x) \in (0,T) \times \bR^d \\
							\label{classic eqn}
&u(0,x)=0.
\end{align}

Consider a Brownian motion $B_t$ in a filtered probability space $(\Omega, \cF_t, \bP)$ with the usual condition.
It is well-known that any predictable function $\sigma(t): \Omega \times (0,T) \to \bR$ satisfying
$$
\int_0^t |\sigma(s)|^2 ds < \infty \quad (a.s.) \quad \forall t \in [0,t],
$$
the It\^o integral 
$$
X_t=\int_0^t \sigma(s) dB_s
$$
 is well-defined and It\^o's formula works for the stochastic process $f(X_t)$ with a smooth function $f$ (cf. \cite[Chapter 5]{krylov2002introduction}).
Moreover, our solution $u$ to equation \eqref{classic eqn} can be derived from the expectation of a composition of a function $f$ and the stochastic process $X_t$. Here is a more explicit statement.
 \begin{theorem}
						\label{existence thm}
Let $T \in (0,\infty)$ and $f$ be a locally integrable function on $(0,T) \times \bR^d$.
Assume that the function $t \in (0,\infty) \mapsto A(t) := \left( a^{ij}(t) \right)_{d \times d}$ is locally integrable, i.e. for each $i$ and $j$,
\begin{align*}
\int_0^T a^{ij}(t) dt < \infty
\end{align*}
and the coefficients $ \left( a^{ij}(t) \right)_{d \times d}$  are nonnegative, i.e. 
$$
a^{ij}(t) \xi^i \xi^j \geq 0 \quad \forall \xi \in \bR^d~ \text{and}~\forall t \in (0,T].
$$
Then there exists a unique solution $u$ to \eqref{classic eqn} and this solution $u$ is given by
\begin{align}
							\label{sol re}
u(t,x)= \int_0^t \bE \left[ f(s,x+X_t -X_s) \right]ds,
\end{align}
where
\begin{align*}
X_t := \sqrt{2} \int_0^t \sqrt{A}^{ij}(s) dB_s^j,
\end{align*}
and $B_t = (B_t^1, \ldots, B_t^d)$ is a $d$-dimensional Brownian motion (Wiener process) and the integral is It\^o's stochastic integral. 
Moreover, for any $p \in [1,\infty]$, we have
\begin{align}
									\label{2022092301}
\|u(t,\cdot)\|_{L_p} \leq \int_0^t \|f(s,\cdot)\|_{L_p} ds \quad \forall t \in [0,T]
\end{align}
and for any functions $h_1$ and $h_2$ on $[0,T]$ which are positive (a.e.), we have
\begin{align}
								\notag
&\int_0^T \|u(t,\cdot)\|^q_{L_p} h_1(t) dt \\
										\label{strong u est}
&\leq \int_0^T\left[  h_1(t)  \left[ \int_0^t |h_2(s)|^{-\frac{1}{q-1}} ds \right]^{q-1} \int_0^t \|f(s,\cdot)\|^q_{L_p} h_2(s)ds\right]  dt .
\end{align}
\end{theorem}
\begin{proof}

\vspace{2mm}
{\bf Part I.} (Uniqueness)
\vspace{2mm}

Even though the coefficients can be unbounded or degenerate,
the uniqueness of a solution can be easily obtained from a classical Fourier transform method with Gr\"onwall's inequality.
To give a rigorous detail, choose a $\varphi$ which is a nonnegative function in $C_c^\infty(\bR^d)$ with a unit integral. 
For $\varepsilon \in (0,1)$, denote 
$$
\varphi^\varepsilon (x) := \frac{1}{\varepsilon^d} \varphi \left( \frac{x}{\varepsilon} \right).
$$
Let $u$ be a solution to
\begin{align*}
&u_t(t,x) 
=  a^{ij}(t)u_{x^ix^j}(t,x) +f(t, x)  \\
&u(0,x)=0
\end{align*}
and define
$$
u^\varepsilon(t,x) = \int_{\bR^d} u(t,y) \varphi^\varepsilon (x-y)dy.
$$
It is sufficient to show that for any $\varepsilon \in (0,1)$ and $t \in (0,T)$, $u(t,x)=0$ for almost every $x \ in \bR^d$. 
Fix $\varepsilon \in (0,1)$, $t \in (0,T)$, and $ x \in \bR^d$. Recalling the definition of a solution and putting $\varphi^\varepsilon(x-\cdot)$ in \eqref{weak formulation}, we have
\begin{align}
						\label{20221012 01}
u^\varepsilon(t,x)= \int_0^t a^{ij}(s) \left(u^\varepsilon \right)_{x^ix^j}(s,x) ds.
\end{align}
For each $\varepsilon \in (0,1)$ and $t \in (0,T)$, \eqref{20221012 01} holds for all $x \in \bR^d$.
Take the $d$-dimensional Fourier transform with respect to $x$ in \eqref{20221012 01} and absolute value. Then we have
\begin{align}
						\label{20221012 02}
\left| \cF[u^\varepsilon(t,\cdot)](\xi) \right|
\leq \int_0^t \left| a^{ij}(s) \xi^i \xi^j \right| \left|\left[\cF[u^\varepsilon (s,\cdot) \right](\xi) \right| ds   
\end{align} 
for all $\xi \in \bR^d$. 
Note that for each $\varepsilon \in (0,1)$ and $\xi \in \bR^d$, \eqref{20221012 02} holds for all $t \in (0,T)$. Thus finally applying  Gr\"onwall's inequality, we have
$$
|\cF[u^\varepsilon(t,\cdot)](\xi)| =0
$$
for all $\varepsilon \in (0,1)$, $t \in (0,T)$, and $\xi \in \bR^d$, which completes the uniqueness of a solution $u$.

\vspace{2mm}
{\bf Part II.} (Existence)
\vspace{2mm}

The existence of a solution $u$ cannot be shown based on a classical Fourier transform method since the coefficients can be degenerate.
Thus we choose a probabilistic method to show the existence of a solution. 
Since it is a well-known fact if the inhomogeneous term $f$ is smooth (even for more general $f$ in a $L_p$-class).
However, it is not easy to find an appropriate reference which exactly fit to our setting (cf. \cite[Section 3]{kim2019sharp}), we give a proof with a detail.
Our main tools are It\^o's formula and a smooth approximation. 
We divide the proof into three steps.
  
 \vspace{2mm}
{\bf Step 1.} (Smooth case) In this step, we assume that for each $t \in (0,T)$, $f(t,x)$ is twice continuously differentiable with respect to $x$.
\vspace{2mm}

Recall that for each $t$, $\left( a^{ij}(t) \right)_{d \times d}$ is a nonnegative symmetric matrix. 
Then there exists a $d \times d$ matrix $\sqrt{A}(t)$ such that 
$$
A(t) =\sqrt{A}(t) \times \sqrt{A}^\ast(t), 
$$
where $\sqrt{A}^\ast$ denotes the transpose matrix of $\sqrt{A}$.
Recall
$$
X_t = \sqrt 2 \int_0^t \sqrt{A}(s) dB_s.
$$
We claim that the function
\begin{align}
								\label{2022101301}
u(t,x):= \int_0^t \bE \left[ f(s,x+X_t -X_s) \right]ds
\end{align}
becomes a solution to \eqref{main eqn}.
As mentioned before,  it is not easy to show that $u$ defined in \eqref{sol re} becomes a solution to \eqref{main eqn} based on an analytic method such as the Fourier transform since the degeneracy of $a^{ij}(t)$ makes the Fourier transform of $u$ lose the integrability. However, it is still possible to apply It\^o's formula. Fix $ s \in (0,T)$ and $x \in \bR^d$.
Apply It\^o's formula to 
$$
f\left(s, x+ \sqrt 2 \int_s^t \sqrt{A}(r)dB_r \right).
$$
Then we have 
\begin{align}
						\notag
f\left(s, x+ \sqrt 2 \int_s^t \sqrt{A}(r)dB_r \right) 
&= f(s,x)+\int_s^t f_{x^i}\left(s, \int_s^\rho \sqrt{A}(r)dB_r \right) dB_\rho \\
						\label{20221012 10}
& \quad + \int_s^t a^{ij}(\rho)f_{x^i x^j}\left(s, x+\sqrt 2\int_s^\rho \sqrt{A}(r)dB_r \right) d\rho
\end{align}
for all $s\leq t < T$ $(a.s)$. 
Taking the expectations in \eqref{20221012 10}, using the property of the It\^o integral that 
$$
\bE\left[\int_s^t f_{x^i}\left(s, \int_s^\rho \sqrt{A}(r)dB_r \right) dB_\rho \right] = 0,
$$
 and recalling the definition of $X_t$, we have
\begin{align}
							\notag
&\bE\left[f\left(s, x+ X_t-X_s \right)\right]
\\
						\label{20221012 11}
&= f(s,x) +\bE\left[\int_s^t a^{ij}(\rho)f_{x^i x^j}\left(s, x+X_\rho - X_s\right) d\rho \right]
\end{align}
for all $0<s \leq t < T$.
Taking the integration $\int_0^t \cdot ~ds$ to both sides of \eqref{20221012 11} and applying the Fubini Theorem, we have
\begin{align*}
&\int_0^t \bE\left[f\left(s, x+ X_t-X_s \right)\right]ds
\\
&= \int_0^t f(s,x) ds 
+\int_0^t \bE\left[\int_s^t a^{ij}(\rho)f_{x^i x^j}\left(s, x+X_\rho - X_s\right) d\rho \right]ds \\
&= \int_0^t f(s,x) ds 
+\int_0^t a^{ij}(\rho) \int_0^\rho \bE\left[ f_{x^i x^j}\left(s, x+X_\rho - X_s\right)  \right]ds d\rho.
\end{align*}
Finally due to the definition of $u$ in \eqref{2022101301}, we have
\begin{align*}
u_t(t,x) =  a^{ij}(t) u_{x^ix^j}(t,x) + f(t,x)
\end{align*}
for all $t \in (0,T)$ and $x \in \bR^d$.

 \vspace{2mm}
{\bf Step 2.} (Bounded case) In this step, we assume that $f$ is bounded. 
\vspace{2mm}

We use Sobolev's mollifiers.
For $\varepsilon \in (0,1)$, denote
$$
f^{\varepsilon} (t,x) = \int_{\bR^d} f(x-\varepsilon y) \varphi (y)dy,
$$
and 
\begin{align*}
u^\varepsilon (t,x) = \int_0^t \bE \left[ f^\varepsilon(s,x+X_t -X_s) \right]ds
\end{align*}
for all $t \in (0,T)$ and $x \in \bR^d$.
Then by the result in Step 1, we have
\begin{align*}
u^\varepsilon_t(t,x) =  a^{ij}(t) u^\varepsilon_{x^ix^j}(t,x) + f^\varepsilon(t,x) .
\end{align*}
In particular, applying the integration by parts, for any $\phi \in C_c^\infty(\bR^d)$, we have
\begin{align*}
(u^\varepsilon(t,\cdot), \phi) 
= \int_0^t \left(u^\varepsilon(s,\cdot), a^{ij}(s)\phi_{x^ix^j} \right) ds +\int_0^t \left(f^\varepsilon(s,\cdot), \phi \right) ds
\quad \forall t \in (0,T).
\end{align*}
Since $f$ is bounded, applying the dominate convergence theorem one can easily check that
$$
u(t,x) 
:=\int_0^t \bE \left[ f(s,x+X_t -X_s) \right]ds
= \limsup_{\varepsilon \downarrow 0} u^\varepsilon(t,x)~(a.e.)
$$
and it becomes a solution to \eqref{classic eqn}.

 \vspace{2mm}
{\bf Step 3.} (General case)
\vspace{2mm}

It suffices to remove the condition that $f$ is bounded.
Due to the linearity of equation \eqref{classic eqn} and the trivial decomposition $f(t,x)=f^+(t,x) - f^-(t,x)$,
we may assume that $f$ is nonnegative, where $f^+(t,x)= \frac{|f(t,x)|+f(t,x)}{2}$ and $f^-(t,x)= \frac{|f(t,x)|-f(t,x)}{2}$.
For $M>0$, define $f^M(t,x) := f(t,x) \wedge M := \min\{ f(t,x) , M\}$ and denote
\begin{align*}
u^M (t,x) = \int_0^t \bE \left[ f^M(s,x+X_t -X_s) \right]ds.
\end{align*}
Then by the result of step 2, for any $M>0$, we have
\begin{align}
								\label{20221013 10}
(u^M(t,\cdot), \phi) 
= \int_0^t \left(u^M(s,\cdot), a^{ij}(s)\phi_{x^ix^j} \right) ds +\int_0^t \left(f^M(s,\cdot), \phi \right) ds
\quad \forall t \in (0,T).
\end{align}
It is obvious that $u^M(t,x) \to u(t,x)$ for all $t \in (0,T)$ and $x \in \bR^d$ as $M \to \infty$.
Finally, taking $M \to \infty$ and applying the monotone and dominate convergence theorems in \eqref{20221013 10}, we show that $u$ is a solution to \eqref{classic eqn}.

\vspace{2mm}
{\bf Part III.} (Estimate)
\vspace{2mm}

We prove \eqref{2022092301} and \eqref{strong u est}. 
By \eqref{sol re}, the generalized Minkowski inequality, and the translation invariant property of the $L_p$-space,
\begin{align*}
\|u(t,\cdot)\|_{L_p} \leq \int_0^t \|f(s,\cdot)\|_{L_p} ds.
\end{align*}
Moreover, applying H\"older's inequality, we have
\begin{align*}
\int_0^T \|u(t,\cdot)\|^q_{L_p} h_1(t) dt
&\leq    \int_0^T h_1(t) \int_0^t \|f(s,\cdot)\|^q_{L_p} h_2(s)ds \left[ \int_0^t |h_2(s)|^{-\frac{1}{q-1}} ds \right]^{q-1} dt.
\end{align*}

\end{proof}

\begin{remark}
Assume that 
$$
\int_0^T \|f(s,\cdot)\|_{L_p} dt <\infty.
$$
Then due to \eqref{2022092301} and the linearity of \eqref{classic eqn}, one can easily find a continuous modification of $u$ so that
\begin{align*}
\sup_{t \in [0,T]} \|u(t,\cdot)\|_{L_p} \leq \int_0^T \|f(s,\cdot)\|_{L_p} ds \quad \forall t \in [0,T].
\end{align*}
\end{remark}

\begin{corollary}
							\label{cor 20220916 01}
Let $T \in (0,\infty)$,  $p \in [1,\infty]$, and $q \in (1,\infty)$.
Suppose that Assumption \ref{main as} holds. 
Additionally, assume that $h_1$ and $h_2$ are functions on $[0,T]$ which are positive (a.e.).
Then for any  locally integrable function $f$ on $(0,T) \times \bR^d$, there is a unique solution $u$ to equation \eqref{main eqn} such that
\begin{align}
							\notag
& \sup_{t \in [0,T]} \left[ \|u(t,\cdot)\|^q_{L_p}e^{-q\int_0^t c(s)ds} \right]   \\
							\label{strong u est 3}
&\leq    \left[ \int_0^T |h_2(t)|^{-\frac{1}{q-1}} dt \right]^{q-1}   \int_0^T e^{-q\int_0^t c(s)ds} \|f(t,\cdot)\|^q_{L_p} h_2(t)dt.
\end{align}
and
\begin{align}
							\notag
&\int_0^T\|u(t,\cdot)\|^q_{L_p}e^{-q\int_0^t c(s)ds} h_1(t) dt  \\
							\label{strong u est 2}
&\leq \int_0^T \left[ h_1(t)   \left[ \int_0^t  |h_2(s)|^{-\frac{1}{q-1}} ds \right]^{q-1}  \int_0^t e^{-q\int_0^s c(\rho)d\rho} \|f(s,\cdot)\|^q_{L_p} h_2(s)ds \right] dt.
\end{align}
\end{corollary}
\begin{proof}

Let $v$ be a solution to the equation 
\begin{align*}
&v_t(t,x)=a^{ij}(t)v_{x^ix^j}(t,x) + e^{-\int_0^t c(s)ds} f\left(t,x-\int_0^t b(s)ds\right),   \\
&v(0,x)=0, 
 \qquad \qquad \qquad \qquad \qquad \qquad \qquad \qquad (t,x) \in (0,T) \times \bR^d.
\end{align*}
Define $U(t,x)= e^{\int_0^t c(s)ds} v\left(t,x+\int_0^t b(s)ds\right)$, where $b(t)= (b^1(t), \ldots, b^d(t))$. 
Then
\begin{align*}
&U_t(t,x)  \\
&= c(t)U(t,x) +  e^{\int_0^t c(s)ds} \left(v_t\left(t,x+\int_0^t b(s)ds\right) +  b^i(t)v_{x^i}\left(t,x+\int_0^t b(s)ds\right)  \right) \\
&=c(t)U(t,x) \\
&\quad +  e^{\int_0^t c(s)ds} \left( a^{ij}(t)v_{x^ix^j}\left(t,x+\int_0^tb(s)ds\right) + e^{-\int_0^t c(s)ds} f(t,x) \right)\\
&\quad +  e^{\int_0^t c(s)ds} \left(b^i(t)v_{x^i}\left(t,x+\int_0^t b(s)ds\right)\right) \\
&=a^{ij}(t)U_{x^ix^j}(t,x) + b^i(t) U_{x^i}(t,x) + c(t) U(t,x)+ f(t,x)
\end{align*}
and
\begin{align*}
U(0,x)=0.
\end{align*}
Thus by the uniqueness of a solution, the solution $u$ to \eqref{main eqn} is given by
$$
u(t,x)=e^{\int_0^t c(s)ds} v\left(t,x+\int_0^t b(s)ds\right)
$$
and obviously
$$
v(t,x)=e^{-\int_0^t c(s)ds} u\left(t,x-\int_0^t b(s)ds\right).
$$
Applying \eqref{strong u est} to $v$ and using the translation invariant property of $L_p$-norms,  we obtain \eqref{strong u est 2}.
Moreover, by \eqref{2022092301} and H\"older's inequality, for any $0\leq t \leq T$, we have
\begin{align*}
&e^{-q\int_0^t c(s)ds}\|u(t,\cdot)\|^q_{L_p} \\
&=\|v(t,\cdot)\|^q_{L_p} \\
&\leq \int_0^t e^{-q\int_0^s c(\rho)d\rho}  \|f(s,\cdot)\|^q_{L_p} h_2(s) ds \left[ \int_0^t |h_2(s)|^{-\frac{1}{q-1}} ds \right]^{q-1} \\
&\leq \int_0^T  e^{-q\int_0^t c(s)ds} \|f(t,\cdot)\|^q_{L_p} h_2(t) dt \left[ \int_0^T |h_2(s)|^{-\frac{1}{q-1}} ds \right]^{q-1},
\end{align*}
which obviously implies \eqref{strong u est 3}.
\end{proof}

\mysection{Estimates for non-degenerate equations}
									\label{section non}

We start the section by reviewing previous weighted estimates with uniform elliptic and bounded coefficients and apply these estimates to our model equation \eqref{classic eqn}.  
We denote 
\begin{align*}
\|f\|_{L_{p,q}(T,w)} 
= \left( \int_0^T \left(\int_{\bR^d} |f(t,x)|^{p} dx \right)^{q/p} w(t)dt \right)^{1/q}.
\end{align*}
As usual, 
$L_{p,q}(T,w)$ denote the spaces of all locally integrable functions $f$ on $(0,T) \times \bR^d$ such that 
$\|f\|_{L_{p,q}(T,w)}  <\infty$.

\begin{theorem}
					\label{classic est 2}
Let $T \in (0,\infty)$, $p,q \in (1,\infty)$, and $w \in A_q (\bR)$.
Assume that the coefficients $a^{ij}(t)$ are uniformly bounded and elliptic, i.e. there exist positive constants $M$ and $\delta$ such that
\begin{align}
							\label{uniform elliptic}
M |\xi|^2 \geq  a^{ij}(t) \xi^i \xi^j \geq \delta |\xi|^2 \quad \forall \xi \in \bR^d.
\end{align}
Then for any $f \in L_{p,q}(T,w)$, there exists a unique solution $u$ to \eqref{classic eqn}
such that
\begin{align}
									\notag
&\left( \int_0^T \left(\int_{\bR^d} |u_{xx}(t,x)|^{p} dx \right)^{q/p} w(t)dt \right)^{1/q} \\
									\label{eqn 20220825 01}
&\qquad \leq N  \left( \int_0^T \left(\int_{\bR^d} |f(t,x)|^{p}  dx \right)^{q/p} w(t)dt \right)^{1/q},
\end{align}
where
$$
N = N\left(p,q,M,\delta,[w]_{A_q(\bR)} \right).
$$
\end{theorem}
\begin{proof}
It is a well-known result which could be easily obtained by combining some classical results.
However, it is not easy to find a paper covering the result directly.
Thus we refer two recent papers \cite[Theorem 2.2]{dong2021approach} handling more general coefficients and \cite[Theorem 2.14]{choi2022weighted} studying time measurable pseudo-differential operators.
\end{proof}

\begin{remark}
Theorem \ref{classic est 2} is enough for our application.
However, as shown in \cite[Theorem 2.2]{dong2021approach} and \cite[Theorem 2.14]{choi2022weighted}, $w_0(x) \in A_p(\bR^d)$ can be inside \eqref{eqn 20220825 01} if \eqref{uniform elliptic} holds. 
In other words, we can find a positive constant $N$ such that
such that
\begin{align*}
&\left( \int_0^T \left(\int_{\bR^d} |u_{xx}(t,x)|^{p} w_0(x)dx \right)^{q/p} w(t)dt \right)^{1/q} \\
&\qquad \leq N  \left( \int_0^T \left(\int_{\bR^d} |f(t,x)|^{p}  w_0(x) dx \right)^{q/p} w(t)dt \right)^{1/q},
\end{align*}
where
$$
N = N\left(p,q,M,\delta,[w]_{A_q(\bR)},[w_0]_{A_p(\bR^d)} \right).
$$
\end{remark}

Next we want to enhance Theorem \ref{classic est 2}.
Specifically, we show the constant $N$ in \eqref{eqn 20220825 01} is independent of the upper bound $M$ of the coefficients $a^{ij}(t)$ and more precise relation between the constant $N$ and the elliptic constant $\delta$. 
However, it seems to be almost impossible to prove it with only analytic tools. 
Thus we recall probabilistic representations of solutions to upgrade Theorem \ref{classic est 2}.
\begin{theorem}
					\label{enhance est thm}
Let $T \in (0,\infty)$, $p,q \in (1,\infty)$, and $w \in A_q (\bR)$.
Assume that the coefficients $a^{ij}(t)$ are uniformly elliptic, i.e. there exists a positive constant $\delta$ such that
\begin{align}
							\label{strong elliptic}
a^{ij}(t) \xi^i \xi^j \geq \delta |\xi|^2 \quad \forall \xi \in \bR^d.
\end{align}
Additionally, we assume that the coefficients $a^{ij}(t)$ are locally integrable, i.e.
\begin{align*}
\int_0^t a^{ij}(s) ds < \infty \qquad \forall t \in (0,T).
\end{align*}
Then for any $f \in L_{p,q}(T,w)$, there exists a unique solution $u$ to \eqref{classic eqn} 
such that
\begin{align}
								\label{enhance est}
 \int_0^T \left(\int_{\bR^d} |u_{xx}(t,x)|^{p} dx \right)^{q/p} w(t)dt 
 \leq \frac{N}{\delta^q}   \int_0^T \left(\int_{\bR^d} |f(t,x)|^{p}  dx \right)^{q/p} w(t)dt,
\end{align}
where 
$$
N = N\left(p,q,[w]_{A_q(\bR)}\right).
$$
\end{theorem}
\begin{proof}
{\bf (Step 1) $a^{ij}(t)u_{x^ix^j}= \delta \Delta u$}.
\vspace{2mm}

For this simple case, we use a basic scaling property of the equation. 
Put $v(t,x) = u(t,\sqrt \delta x)$.
Since $u$ is the solution to 
\begin{align*}
&u_t(t,x) 
=  \delta \Delta u(t,x) +f(t, x)  \\
&u(0,x)=0,
\end{align*}
we have
\begin{align*}
&v_t(t,x) 
=   \Delta v(t,x) +f(t, \sqrt \delta x)  \\
&v(0,x)=0.
\end{align*}
Thus applying \eqref{eqn 20220825 01}, we have
\begin{align*}
&\left( \int_0^T \left(\int_{\bR^d} |v_{xx}(t,x)|^{p}  dx \right)^{q/p} w(t)dt \right)^{1/q} \\
&\qquad \leq N  \left( \int_0^T \left(\int_{\bR^d} |f(t, \sqrt \delta  x)|^{p} dx \right)^{q/p} w(t)dt \right)^{1/q},
\end{align*}
where
$$
N = N\left(p,q,[w]_{A_q(\bR)} \right).
$$
Finally, we obtain \eqref{enhance est} by the simple change of the variable $\sqrt \delta x \to x$.

\vspace{2mm}
{\bf (Step 2)} General $a^{ij}(t)u_{x^ix^j}$.
\vspace{2mm}

To prove a general case, we use probabilistic solution representations. 
We may assume that 
$$
\int_0^T a^{ij}(t) dt < \infty
$$
since the constant $N$ in \eqref{enhance est} is independent of $T$.
Additionally, due to the trivial constant extension $a^{ij}(t) 1_{ t \in (0,T)} + a^{ij}(T) 1_{ t  \geq T}$, 
we may assume that $a^{ij}(t)$ is defined on $(0,\infty)$.
Consider two independent $d$-dimensional Brownian motions $B_t$ and $W_t$
in a probability space $(\Omega, \cF_t, \bP)$.
Set
$$
\left( a^{ij}(t) \right)_{d \times d} =A(t) =\sqrt{A}(t) \times \sqrt{A}^\ast(t),
$$
\begin{align*}
X_t := \sqrt{2} \int_0^t \sqrt{A}^{ij}(s) dB_s^j,
\end{align*}
\begin{align*}
X^2_t := \sqrt{2} \int_0^t \left(\sqrt{A(s) - \delta {I}}^{ij} \right) dB_s^j,
\end{align*}
\begin{align*}
X^1_t := \sqrt{2} \sqrt \delta {I}^{ij} W^j_t,
\end{align*}
where $I=(I^{ij})_{d \times d}$ denotes the $d$ by $d$ identity matrix whose diagonal entries are 1 and the other entries are zero and
$\sqrt{{A}(s) - \delta {I}}$ is a matrix so that 
$$
\sqrt{{A}(s) - \delta {I}} \sqrt{{A}(s) - \delta {I}} = A(s)-\delta I,
$$
which exists due to \eqref{strong elliptic}, i.e. $A(s)-\delta I$ is a nonnegative symmetric matrix.
Then due to \eqref{sol re}, the solution $u$ is given by
\begin{align}
									\notag
u(t,x)
&= \int_0^t \bE \left[ f(s,x+X_t -X_s) \right]ds  \\
									\label{eqn 20220722 10}
&= \int_0^t \bE \left[ f(s,x+X^1_t -X^1_s + X^2_t -X^2_s) \right]ds,
\end{align}
where the last equality is due to the fact that two probabilistic distributions of $X_t -X_s$ and $X^1_t -X^1_s + X^2_t -X^2_s$ are equal for all $0<s<t$. 
Moreover, due to the independence of two Brownian motions $B_t$ and $W_t$, we can split the random parameters in \eqref{eqn 20220722 10}.
Additionally, applying Fubini's theorem we have
\begin{align}
u(t,x)
							\notag
&=\int_0^t \bE \left[ f(s,x+X^1_t -X^1_s + X^2_t -X^2_s) \right]ds \\
							\notag
&=\int_0^t \bE' \left[ \bE \left[ f(s,x+X^1_t(\omega) -X^1_s(\omega) + X^2_t(\omega') -X^2_s(\omega')) \right] \right]ds \\
							\label{eqn 20220722 100}
&=\bE'\left[ \int_0^t  \bE \left[ f(s,x+X^1_t(\omega) -X^1_s(\omega) + X^2_t(\omega') -X^2_s(\omega')) \right]ds \right].
\end{align}
For each fixing $\omega'$, 
the function 
\begin{align*}
v^{\omega'}(t,x) :=\int_0^t  \bE \left[ f(s,x+X^1_t(\omega) -X^1_s(\omega)  -X^2_s(\omega')) \right]ds
\end{align*}
becomes a solution to the equation 
\begin{align*}
&v^{\omega'}_t(t,x) 
=   \delta \Delta v^{\omega'}(t,x) +f(t, x-X^2_t(\omega'))  \\
&v^{\omega'}(0,x)=0.
\end{align*}
Thus by the result in {\bf Step 1},
\begin{align}
											\label{eqn 20220725 01}
\int_0^T  \left(\int_{\bR^d} | v^{\omega'}_{xx}(t,  x)|^p  dx \right)^{q/p} w(t)dt 
\leq \frac{N}{\delta^q}\int_0^T  \left(\int_{\bR^d} |f(t, x-X^2_t(\omega'))|^p  dx \right)^{q/p} w(t)dt,
\end{align}
where  $N$ depends only on $p$, $q$, $[w]_{A_q(\bR)}$, and $\kappa$.
Moreover, by \eqref{eqn 20220722 100},
\begin{align}
							\label{v  repre}
u_{xx}(t,x)=\bE'\left[  v_{xx}^{\omega'}\left(t,x + X_t^2(\omega') \right) \right].
\end{align}
Finally applying \eqref{v repre}, \eqref{eqn 20220725 01}, the generalized Minkowski's inequality, and Jensen's inequality, we have
\begin{align*}
&\int_0^T  \left(\int_{\bR^d} | u_{xx}(t,  x)|^p  dx \right)^{q/p} w(t)dt \\
&\leq N\bE' \left[  \int_0^T  \left(\int_{\bR^d} | v^{\omega'}_{xx}(t,  x + X_t^2(\omega'))|^p dx \right)^{q/p} w(t)dt \right] \\
&\leq \frac{N}{\delta^q}\int_0^T  \left(\int_{\bR^d} |f(t, x)|^p w\left(x+k(t)\right) dx \right)^{q/p} w(t)dt.
\end{align*}
\end{proof}

\begin{remark}
						\label{fail method}
We hope that there is a positive constant $N$ such that
such that
\begin{align*}
&\left( \int_0^T \left(\int_{\bR^d} |u_{xx}(t,x)|^{p} dx \right)^{q/p} w(t)dt \right)^{1/q} \\
&\qquad \leq \frac{N}{\delta^q}  \left( \int_0^T \left(\int_{\bR^d} |f(t,x)|^{p}   dx \right)^{q/p} w(t)dt \right)^{1/q},
\end{align*}
where
$$
N = N\left(p,q,[w]_{A_q(\bR)},[w_0]_{A_p(\bR^d)} \right).
$$
However, it cannot be obtained by following the proof of Theorem \ref{enhance est thm} since 
$$
\int_{\bR^d}  |f(t, x-X^2_t(\omega'))|^p dx =  \int_{\bR^d} | f(t,  x )|^p dx \quad \forall \omega'~ \text{and}~ \forall t
$$
is used in the proof. 
\end{remark}

\mysection{Proof of the main theorem}
							\label{pf main thm}

\vspace{2mm}
{\bf Proof of Theorem \ref{main thm}}
\vspace{2mm}

Due to Theorem \ref{existence thm}, the existence and uniqueness of a solution $u$ is obvious. 
Moreover, \eqref{main particular est -1}, \eqref{main particular est 0}, and \eqref{main particular est} 
can be easily obtained from \eqref{main est -1}, \eqref{main est 0}, and \eqref{main est} 
since $|t|^{\beta} \in A_q(\bR)$ for any $-1 < \beta_1 < q-1$ (see \cite[Example 7.1.7]{grafakos2014classical}).
Thus it suffices to show \eqref{main est -1}, \eqref{main est 0} and \eqref{main est}.
Let $u$ be the solution to \eqref{main eqn}. 
First we show \eqref{main est -1} and \eqref{main est 0}.
For each $\varepsilon \in (0,1)$, we denote
$$
h_{1,\varepsilon}(t)= w(\alpha(t) + \varepsilon t ) \left(\delta(t) + \varepsilon \right)
$$
and
$$
h_{2,\varepsilon}(t)= w(\alpha(t) + \varepsilon t) |\delta(t) + \varepsilon|^{1-q}.
$$
Then by \eqref{strong u est 3} and \eqref{strong u est 2} with a simple change of variable,
\begin{align*}
& \sup_{t \in [0,T]} \left[ \|u(t,\cdot)\|^q_{L_p}e^{-q\int_0^t c(s)ds} \right]   \\
&\leq    \left[ \int_0^T \left|w(\alpha(t) + \varepsilon t) |\delta(t) + \varepsilon|^{1-q}\right|^{-\frac{1}{q-1}} dt \right]^{q-1}    \\
&\quad \times \int_0^T \|f(t,\cdot)\|^q_{L_p} e^{-q\int_0^t c(s)ds}   w(\alpha(t) + \varepsilon t) |\delta(t) + \varepsilon|^{1-q}(t)dt \\
&\leq    \left[ \int_0^{\alpha(T)+\varepsilon T} |w(t) |^{-\frac{1}{q-1}}dt \right]^{q-1}  \\
&\quad \times \int_0^T \|f(t,\cdot)\|^q_{L_p} e^{-q\int_0^t c(s)ds} w(\alpha(t) + \varepsilon t) |\delta(t) + \varepsilon|^{1-q}(t)dt 
\end{align*}
and
\begin{align*}
&\int_0^T\|u(t,\cdot)\|^q_{L_p}e^{-q\int_0^t c(s)ds} w(\alpha(t) + \varepsilon t ) \left(\delta(t) + \varepsilon \right) dt  \\
&\leq \left[\int_0^T w(\alpha(t) + \varepsilon t ) \left(\delta(t) + \varepsilon \right)  \left[ \int_0^t |w(\alpha(s) + \varepsilon s) |\delta(s + \varepsilon)|^{1-q}|^{-\frac{1}{q-1}} ds \right]^{q-1}  dt\right] \\
&\quad \times \int_0^T \|f(t,\cdot)\|^q_{L_p} e^{-q\int_0^t c(s)ds} w(\alpha(t) + \varepsilon t) |\delta(t + \varepsilon)|^{1-q}dt.
\end{align*}
Moreover, by taking $\varepsilon \to 0$, we have
\begin{align*}
& \sup_{t \in [0,T]} \left[ \|u(t,\cdot)\|^q_{L_p}e^{-q\int_0^t c(s)ds} \right]   \\
&\leq    \left[ \int_0^{\alpha(T)} |w(t) |^{-\frac{1}{q-1}}dt \right]^{q-1} \int_0^T \|f(t,\cdot)\|^q_{L_p} e^{-q\int_0^t c(s)ds} w(\alpha(t)) |\delta(t)|^{1-q}(t)dt 
\end{align*}
and
\begin{align}
							\notag
&\int_0^T\|u(t,\cdot)\|^q_{L_p}e^{-q\int_0^t c(s)ds} w(\alpha(t) ) \left(\delta(t) \right) dt  \\
							\notag
&\leq \left[\int_0^T w(\alpha(t)  ) \left(\delta(t) \right)  \left[ \int_0^t |w(\alpha(s) ) |\delta(s)|^{1-q}|^{-\frac{1}{q-1}} ds \right]^{q-1}  dt\right] \\
							\label{eqn 20220916 30}
&\quad \times \int_0^T \|f(t,\cdot)\|^q_{L_p} e^{-q\int_0^t c(s)ds} w(\alpha(t) ) |\delta(t)|^{1-q}dt.
\end{align}
One may think that this limit procedure does not seem to be clear. However, it is clear if our weight $w$ is continuous.
Moreover, if $w$ is bounded, then $w$ can be approximated by a sequence of continuous functions with a uniform upper bound.
Finally, considering $w \wedge M$ for any positive constant $M>0$, we can complete the limit procedure due to the monotone convergence theorem as $M \to \infty$.

We keep going to estimate the term in the middle of \eqref{eqn 20220916 30}.
Recalling the definition of $[w]_{A_p(\bR)}$ and applying the change of variable $\alpha(t):=\int_0^t \delta(s)ds \to t$, we have
\begin{align*}
&\int_0^T w(\alpha(t)  ) \delta(t) \left[ \int_0^t |w(\alpha(s) ) |\delta(s)|^{1-q}|^{-\frac{1}{q-1}} ds \right]^{q-1}  dt \\
&\leq \int_0^T w(\alpha(t)  )\delta(t)   dt   \left[ \int_0^T |w(\alpha(t) )|^{-\frac{1}{q-1}} \delta(t) ds \right]^{q-1} \\
&\leq \int_0^{\alpha(T)} w(t)    dt   \left[ \int_0^{\alpha(T)} |w(t )|^{-\frac{1}{q-1}} ds \right]^{q-1} \\
&\leq  [w]_{A_p(\bR)} \left[\alpha(T) \right]^q .
\end{align*}
By putting the above computations in \eqref{eqn 20220916 30}, we obtain \eqref{main est 0}.

\vspace{2mm}
Next we prove \eqref{main est}. We may assume that $f$ has a compact support in $[0,T] \times \bR^d$. 
We divide the proof into several steps.

\vspace{2mm}
{\bf (Step 1)} $\delta(t) \geq \varepsilon$ and $b^i(t)=c(t)=0$ for all $i$ and $t$.
\vspace{2mm}

We first assume that there exists a positive constant $\varepsilon \in (0,1)$ such that $\delta(t) \geq \varepsilon$ for all $t$. 
Additionally, suppose that $b^i(t)=0$ and $c(t)=0$ for all $t$ and $i$ in this first step.
Denote 
$$
\alpha(t) = \int_0^t \delta(s)ds.
$$
Then $\beta(t)$ becomes a strictly increasing function and it has the inverse $\beta(t):[0,\infty) \to [0,\infty)$ such that
\begin{align}
						\label{beta derivative}
\beta'(t) = \frac{1}{\alpha'(\beta(t))} = \frac{1}{ \delta( \beta(t))} \quad \forall t \in [0,\infty).
\end{align}
Define $v(t,x)=u(\beta(t),x)$. Then since $u$ is a solution to \eqref{main eqn},
\begin{align*}
v_t(t,x) 
= u_t(\beta(t),x) \beta'(t) 
= \frac{a^{ij}(\beta(t))}{\delta(\beta(t))} v_{x^ix^j}(t,x) + \frac{f(\beta(t),x)}{\delta(\beta(t))}
\end{align*}
and $v(0,x)=0$.
Note that 
$$
 \frac{a^{ij}(\beta(t))}{\delta(\beta(t))}   \xi^i \xi^j \geq |\xi|^2 \quad \forall \xi \in \bR^d.
$$
In other words,  $v$ becomes the solution to
\begin{align}
				\notag
&v_t(t,x)=\tilde  a^{ij}(t)v_{x^ix^j}(t,x) + \frac{f(\beta(t),x)}{\delta(\beta(t))}  \qquad (t,x) \in (0,T) \times \bR^d,   \\
&u(0,x)=0, 
				\label{v eqn}
\end{align}
with the coefficients $\tilde a^{ij}(t) = \frac{a^{ij}(\beta(t))}{\delta(\beta(t))}$ whose elliptic constant is  $1$.
Moreover, it is obvious that $\tilde a^{ij}(t)$ is locally integrable. Indeed,
by the change of the variable $\beta(t) \to t$ and \eqref{beta derivative},  
\begin{align*}
\int_0^T \tilde a^{ij}(t) dt = \int_0^{\beta(T)} a^{ij}(t)dt < \infty.
\end{align*}
Thus applying \eqref{enhance est}, we have
\begin{align}
							\notag
&\left( \int_0^{T_0} \left(\int_{\bR^d} |v_{xx}(t,x)|^{p}  dx \right)^{q/p} w(t)dt \right)^{1/q} \\
							\label{v trans est}
&\qquad \leq N  \left( \int_0^{T_0} \left(\int_{\bR^d} \left|\frac{f(\beta(t),x)}{\delta(\beta(t))}\right|^{p}  dx \right)^{q/p} w(t)dt \right)^{1/q},
\end{align}
where 
$$
N = N\left(p,q,[w_0]_{A_q(\bR)}, \kappa\right)
$$
and $T_0$ is a constant so that $\beta(T_0)= T$. 
By considering the change of variables $\beta(t) \to t$ in \eqref{v trans est}, we finally obtain
\begin{align}
								\notag
&\left( \int_0^{T} \left(\int_{\bR^d} |u_{xx}(t,x)|^{p}  dx \right)^{q/p} w(\alpha(t))   \delta(t) dt \right)^{1/q} \\
								\label{step1 est}
&\qquad \lesssim  \left( \int_0^{T} \left(\int_{\bR^d} |f(t,x)|^{p}  dx \right)^{q/p} w(\alpha(t)) (\delta(t))^{1-q}  dt \right)^{1/q}.
\end{align}

\vspace{2mm}
{\bf (Step 2)} $b^i(t)=c(t)=0$ for all $i$ and $t$.
\vspace{2mm}

In this step, we remove the condition $\delta(t) \geq \varepsilon$. 
For any $\varepsilon \in (0,1)$, we can rewrite \eqref{main eqn} as
\begin{align*}
&u_t(t,x)=(a^{ij}(t)+\varepsilon I_{d \times d})u_{x^ix^j}(t,x) +f(t,x)- \varepsilon \Delta u,   \\
&u(0,x)=0, 
 \qquad \qquad \qquad \qquad \qquad \qquad \qquad \qquad (t,x) \in (0,T) \times \bR^d,
\end{align*}
where $I_{d \times d}$ denotes the $d$ by $d$ identity matrix whose diagonal entries are 1 and the other entries are zero.
Thus applying \eqref{step1 est}, we have
\begin{align}
						\notag
&\left( \int_0^{T} \left(\int_{\bR^d} |u_{xx}(t,x)|^{p}  dx \right)^{q/p} w(\alpha_\varepsilon(t))   (\delta(t)+\varepsilon) dt \right)^{1/q} \\
						\notag
&\qquad \lesssim  \left( \int_0^{T} \left(\int_{\bR^d} |f(t,x)|^{p}  dx \right)^{q/p} w(\alpha_\varepsilon(t)) (\delta(t)+\varepsilon)^{1-q}  dt \right)^{1/q} \\
						\label{eqn 20220825 30}
&\qquad  \quad +\left( \int_0^{T} \left(\int_{\bR^d} |\varepsilon \Delta u (t,x)|^{p}  dx \right)^{q/p} w(\alpha_\varepsilon(t)) (\delta(t)+\varepsilon)^{1-q}  dt \right)^{1/q},
\end{align}
where $\alpha_\varepsilon(t) = \int_0^t (\delta(s) + \varepsilon)ds$. 
Observe that 
\begin{align*}
& \int_0^{T} \left(\int_{\bR^d} |\varepsilon \Delta u (t,x)|^{p} w(x+k(t)) dx \right)^{q/p} w_0(\alpha_\varepsilon(t)) (\delta(t)+\varepsilon)^{1-q}  dt \\
&= \int_0^{T} \left(\int_{\bR^d} | \Delta u (t,x)|^{p} w(x+k(t)) dx \right)^{q/p} w_0(\alpha_\varepsilon(t))(\delta(t)+\varepsilon) \left(\frac{\varepsilon}{\delta(t)+\varepsilon} \right)^q dt,
\end{align*}
\begin{align*}
(\delta(t)+\varepsilon) \left(\frac{\varepsilon}{\delta(t)+\varepsilon}\right)^q
\leq (\delta(t))^{1-q}
\end{align*}
and
\begin{align*}
(\delta(t)+\varepsilon) \left(\frac{\varepsilon}{\delta(t)+\varepsilon}\right)^q \to 0~\text{as}~ \varepsilon \to 0,
\end{align*}
where $0^{1-q} := \infty$.
Thus due to the dominate convergence theorem and the definition of the integral in \eqref{improper integ}, taking $\varepsilon \to 0$ in \eqref{eqn 20220825 30}, we have
\begin{align}
							\notag
& \int_0^{T} \left(\int_{\bR^d} |u_{xx}(t,x)|^{p}  dx \right)^{q/p} w_0(\alpha(t))   \delta(t) dt \\
							\label{step2 est}
&\lesssim  \int_0^{T} \left(\int_{\bR^d} |f(t,x)|^{p}  dx \right)^{q/p} w_0(\alpha(t)) (\delta(t))^{1-q}  dt. 
\end{align}

\vspace{2mm}
{\bf (Step 3)} (General case).
\vspace{2mm}

Let $v$ be a solution to the equation 
\begin{align*}
&v_t(t,x)=a^{ij}(t)v_{x^ix^j}(t,x) + e^{-\int_0^t c(s)ds} f\left(t,x-\int_0^t b(s)ds\right),   \\
&v(0,x)=0, 
 \qquad \qquad \qquad \qquad \qquad \qquad \qquad \qquad (t,x) \in (0,T) \times \bR^d.
\end{align*}
The as shown in the proof of Corollary \ref{cor 20220916 01}, the solution $u$ is given by 
$$
u(t,x)=e^{\int_0^t c(s)ds} v\left(t,x+\int_0^t b(s)ds\right)
$$
and obviously
$$
v(t,x)=e^{-\int_0^t c(s)ds} u\left(t,x-\int_0^t b(s)ds\right).
$$
Applying \eqref{step2 est} to $v$, we have
\begin{align*}
							\notag
& \int_0^{T} \left(\int_{\bR^d} \left|e^{-\int_0^t c(s)ds} u_{xx}\left(t,x-\int_0^t b(s)ds\right)(t,x)\right|^{p}  dx \right)^{q/p} \\
& \quad \times  w(\alpha(t))   \delta(t) dt \\
							\label{step2 est}
&\lesssim  \int_0^{T} \left(\int_{\bR^d} \left|f\left(t,x-\int_0^t b(s)ds\right)\right|^{p}  dx \right)^{q/p} e^{-q\int_0^t c(s)ds} w(\alpha(t)) (\delta(t))^{1-q}  dt. 
\end{align*}
Finally, the translation $x \to x+\int_0^t b(s)ds$ leads us to \eqref{main est}.

\qed

\mysection{Acknowledgement}

I would like to thank prof. Kyeong-Hun Kim for careful reading and suggesting valuable comments.

\vspace{2mm}
{\bf Data Availability}
\vspace{2mm}

Data sharing not applicable to this article as no datasets were generated or analysed during the current study.

\end{document}